\documentclass[10pt,reqno]{amsart}

\usepackage{amssymb}
\usepackage{amsfonts}
\usepackage{amsthm}
\usepackage{amsmath}
\usepackage{bbm}
\usepackage{xcolor}
\usepackage{mathtools}
\usepackage{bigints}
\usepackage{enumerate}
\usepackage[normalem]{ulem}               % to striketrhourhg text
\newcommand\redout{\bgroup\markoverwith
{\textcolor{red}{\rule[0.5ex]{2pt}{0.8pt}}}\ULon}

\numberwithin{equation}{section}
\allowdisplaybreaks

\newtheorem{theorem}{Theorem}[section]

\newtheorem{prop}[theorem]{Proposition}
\newtheorem{lemma}[theorem]{Lemma}

\theoremstyle{definition}

\newtheorem{exmp}[theorem]{Example}

\theoremstyle{remark}

\newcommand{\R}{\mathbb{R}}

\newcommand{\HH}{\mathbb{H}}
\newcommand{\I}{\mathbb{I}}

\renewcommand{\hat}{\widehat}
\newcommand{\eps}{\varepsilon}

\newcommand{\scriptM}{\mathcal{M}}

\newcommand{\scriptT}{\mathcal{T}}

\begin{document}
\title{Inequalities of Brascamp--Lieb type on the Heisenberg group}
\author{Kaiyi Huang}
\author{Betsy Stovall}
\address{University of Wisconsin--Madison}
\email{khuang247@wisc.edu}
\email{stovall@math.wisc.edu}

\begin{abstract}
    We prove both necessary and sufficient conditions for $L^p$-bound\-ed\-ness of certain multilinear generalized Radon transforms that arise as Heisenberg group analogues of the Brascamp--Lieb inequalities on Euclidean space. The necessary and sufficient conditions coincide in some important special cases, but differ in others.
\end{abstract}

\maketitle

\section{Introduction}

In recent work,  F\"assler--Pinamonti \cite{FPmathz} and F\"assler--Orponen--Pinamonti \cite{FOParxiv} introduced a multilinear form and corresponding (sharp) $L^p$ inequality that is the natural Heisenberg group analogue of the Loomis--Whitney inequality.  In this article, we take this generalization one step further and establish $L^p$ inequalities for a large class of multilinear forms that arise as the Heisenberg analogues of the Brascamp--Lieb inequalities on Euclidean space.   

More concretely, we work in the Heisenberg group $\HH^n$, which we identify with $\R^n \times \R^n \times \R$, endowed with the group operation $(x,y,t) \bullet (x',y',t') = (x+x',y+y',t+\tfrac12(x\cdot y' - y\cdot x'))$.  Given $V \leq \R^n$ (for $W$ a vector space, $V \leq W$ means $V$ is a vector subspace of $W$), we let $L$ denote orthogonal projection to $V$ and define \textit{vertical projections} $\pi_L:\HH^n \to V \times \R^n \times \R$, $\tilde \pi_L:\HH^n \to \R^n \times V \times \R$ by the formulae
\begin{gather*}
\pi_L(x,y,t) := (x,y,t)\bullet (\hat L x,0,0)^{-1} = (Lx,y,t+\tfrac12(\hat Lx) \cdot (\hat Ly)), \\
\tilde\pi_L(x,y,t) := (x,y,t) \bullet (0,\hat Ly,0)^{-1} = (x,Ly,t-\tfrac12(\hat Lx) \cdot (\hat Ly)),
\end{gather*}
where $\hat L:=\mathbb I_n - L$ is orthogonal projection to $V^\perp$.  
Recalling the Brascamp--Lieb inequalities, which were introduced in \cite{BrasLiebAdvances73}, and whose boundedness was characterized in \cite{BCCTshort, BCCTlong, Lieb}, we are led to consider  $(m+m^\prime)$-tuples $(L_j)_{j=1}^{m+m^\prime}$, $L_j:\R^n \to V_j$, and a corresponding multilinear form 
\begin{equation} \label{E:main ml form}
\scriptM(f_1,\ldots,f_{m+m^\prime}) := \int_{\HH^n} \prod_{j=1}^{m+m^\prime} f_j \circ \pi_j(x,y,t) \, dx\, dy\, dt,
\end{equation}
where we use the notation
\begin{equation} \label{E:symmetry}
\pi_j:=\pi_{L_j}, \quad 1 \leq j \leq m,\qquad \pi_{j+m}:=\tilde \pi_{L_{j+m}}, \quad 1 \leq j \leq m^\prime.
\end{equation}
The purpose of this article is to establish both necessary and sufficient conditions on exponents $p_1,\ldots,p_{m+m^\prime} \in [1,\infty]$ such that the inequality 
\begin{equation} \label{E:Lp ineq}
\scriptM(f_1,\ldots,f_{m+m^\prime}) \lesssim \prod_{j=1}^m \|f_j\|_{L^{p_j}(V_j \times \R^n \times \R)} \times\prod_{j=m+1}^{m+m^\prime}\|f_{j}\|_{L^{p_{j}}(\R^n \times V_j \times \R)},
\end{equation}
holds for all nonnegative, continuous functions $f_1,\ldots,f_{m+m^\prime}$ with compact support, with implicit constant independent of the $f_k$.  

In fact, because inserting additional terms $f_j \circ \pi_j$ with corresponding exponents $p_j=\infty$ makes no difference in the validity of \eqref{E:Lp ineq}, the problem immediately reduces to establishing necessary and sufficient conditions for boundedness of the form in \eqref{E:main ml form} with $m=m^\prime$, $L_j=L_{j+m}$ for $1\leq j\leq m$.  (Due to the notational simplifications, we state our results below for this superficially less general form.)  

In Section~\ref{S:necessity}, we will prove the following necessary conditions on the $p_k$ for validity of (the above-mentioned reduced form of) \eqref{E:Lp ineq}.

\begin{theorem}\label{T:nec}
Given $m$ and vector spaces $V_1,\ldots,V_{2m} \leq \R^n$, with $V_j=V_{j+m}$, $1 \leq j \leq m$, if the inequality  
\begin{equation} \label{E:Lp ineq sym}
\scriptM(f_1,\ldots,f_{2m}) \lesssim \prod_{j=1}^m \|f_j\|_{L^{p_j}(V_j \times \R^n \times \R)}\|f_{j+m}\|_{L^{p_{j+m}}(\R^n \times V_j \times \R)},
\end{equation}
holds uniformly over all nonnegative, continuous $f_1,\ldots,f_{2m}$, then $p_j \geq 1$, $1 \leq j \leq 2m$, and the following  conditions all hold for all subspaces $V \leq \R^n$ and $W \leq V^\perp$:
\begin{gather}
\tag{A}\label{A} n+1=\sum_{j=1}^{m}\frac{\dim V_j+1}{p_j}+\frac{n+1}{p_{j+m}}=\sum_{j=1}^{m}\frac{n+1}{p_j}+\frac{\dim V_j+1}{p_{j+m}};\\
\tag{B1}\label{B1} \dim V+1 \leq \sum_{j=1}^{m}\frac{\dim L_jV+1}{p_j}+\frac{\dim V+1}{p_{j+m}};\\
\tag{B2}\label{B2} \dim V+1 \leq \sum_{j=1}^{m}\frac{\dim V+1}{p_j}+\frac{\dim L_jV+1}{p_{j+m}};\\
\tag{C1}\label{C1}
\dim W^\perp - \dim V \geq\displaystyle\sum_{j=1}^m\frac{\dim W^\perp-\dim L_jV}{p_j}+\frac{\dim(L_jW)^\perp-\dim V}{p_{j+m}};\\
\tag{C2}\label{C2} \dim W^\perp - \dim V \geq\displaystyle\sum_{j=1}^m\frac{\dim(L_jW)^\perp-\dim V}{p_j}+\frac{\dim W^\perp-\dim L_j V}{p_{j+m}}.
\end{gather}
\end{theorem}
Here, for any subspace $V\leq\R^n$, $V^\perp$ is the orthogonal complement of $V$ in $\R^n$, while $(L_jV)^\perp$ is the orthogonal complement of $L_jV$ in $L_j(\R^n)=V_j$.

We see that $\dim(L_j V^\perp)^\perp \leq \dim(L_j V)$, and in the case of equality, \eqref{C1} and \eqref{C2} (with $W=V^\perp$) are equivalent to 
\begin{gather}
\tag{C}\label{C}
\sum_{j=1}^m (\dim V - \dim L_j V)\left(\tfrac1{p_j}-\tfrac1{p_{m+j}}\right)=0, \qquad V \leq \R^n,
\end{gather}
which is stronger than \eqref{C1} and \eqref{C2} in the more general case $W \leq V^\perp$.

Our main theorem is that, under the additional hypothesis \eqref{C}, the conditions above are also sufficient.  

\begin{theorem}\label{T:main}
Given $V_1,\ldots,V_{2m} \leq \R^n$, obeying the condition $V_j=V_{j+m}$, and exponents $p_1,\ldots,p_{2m} \in [1,\infty]$, obeying \eqref{A}, \eqref{B1}, \eqref{B2}, and \eqref{C}, Inequality \eqref{E:Lp ineq sym} holds uniformly over all nonnegative, continuous $f_1,\ldots,f_{2m}$.  
\end{theorem}

In Section~\ref{S:examples}, we will prove that \eqref{C} is not necessary in general, though Theorem~\ref{T:main} is sharp when the $L_j$ are coordinate projections whose kernels direct sum to $\R^n$. Determining sharp necessary and sufficient conditions for \eqref{E:Lp ineq sym} seems like an interesting question and, furthermore, there are other natural generalizations one might consider, such as replacing $L_j\oplus \mathbb I_n$ with a more general linear map,  replacing $\HH^n$ with a more general nilpotent Lie group, or by inserting a cutoff function into the integral defining \eqref{E:main ml form}, but we do not address these generalizations here.    

\subsection*{Prior results} Our results fall into a broad body of literature establishing necessary and sufficient conditions for boundedness of operators, linear and multilinear, lying within various classes of generalized Radon transforms. 
Extensive histories and bibliographies describing this prior work may be found in \cite{TaoWright} and  \cite{GressmanTesting}, which also represent important milestones in the general theory, but here we focus on those results most directly relevant to the present article.  

Considering a general (localized) multilinear form  
\begin{equation} \label{E:general T}
\scriptT(f_1,\ldots,f_m):=\int_{\R^n} \prod_{j=1}^m f_j \circ \pi_j(x)\, a(x)\,dx,
\end{equation}
with the $\pi_j:\R^n \to \R^{n_j}$ smooth submersions and $a$ (say) the characteristic function of a ball, optimal or near-optimal Lebesgue space bounds are known if $n_j=n-1$ for all $j$ \cite{StovallMultilin, TaoWright}, or if the $\pi_j$ are all linear \cite{BCCTshort}; optimal bounds are also known in a wide variety of other circumstances, but the conditions on the $\pi_j$ become more restrictive.  
Under an additional scaling condition on the Lebesgue exponents, namely, that $n=\sum_j \frac{n_j}{p_j}$, 
Bennett--Bez--Buschenhenke--Cowling--Flock \cite{BBBCF} have shown that $L^p$ bounds on \eqref{E:general T} in the case of linear $\pi_j$ remain valid under small $C^\infty$ perturbations of the $\pi_j$, and Gressman has shown \cite{GressmanTesting} equivalence of $L^p$ bounds with a nontrivial testing condition.  

It is thus natural to consider multilinear generalized Radon transforms in which the underlying maps $\pi_j$ are nonlinear, the dimensions of their codomains vary, or the above-mentioned scaling condition fails; the present setting is a relatively simple instance in which all of these conditions are met.  Moreover, there has been significant recent interest in establishing isoperimetric and other geometric inequalities in the Heisenberg group.  (See, for instance, \cite{balogh2018geometric, fassler2016improved, FOParxiv, ruzhansky2020geometrid}.)  

Indeed, as mentioned previously, our starting point is recent work of F\"assler--Orponen--Pinamonti \cite{FOParxiv}, who, motivated by questions arising in the geometric measure theory of Heisenberg groups, had introduced the ``Loomis--Whitney'' variant of \eqref{E:main ml form}, in which
\begin{equation} \label{E:LoomisWhitneyCase}
L_j(x_1,\ldots,x_n) := (x_1,\ldots,\widehat{x_j},\ldots,x_n), \qquad 1 \leq j \leq n,  
\end{equation}
and F\"assler--Pinamonti had established sharp $L^p$ bounds in \cite{FPmathz}.  
Though the form \eqref{E:main ml form}, with maps \eqref{E:LoomisWhitneyCase}, and its sharp or sharp-to-endpoints $L^p$ inequalities  had implicitly appeared in prior work \cite{Littman73, StovallMultilin}, it had not previously been recognized as a variant of the Loomis--Whitney inequality, and this arrangement naturally raises the question of generalizations in the spirit of the Brascamp--Lieb inequalities.  Some related questions were posed in   \cite{bramati2019geometric} in a wider class of groups (graded nilpotent Lie algebras). (See also \cite{FPmathz}.)

\subsection*{Outline}  In Section~\ref{S:examples}, we develop a few examples illustrating cases in which Theorem~\ref{T:main} is and is not sharp. In Section~\ref{S:necessity}, we prove the necessity result, Theorem~\ref{T:nec}.  In Section~\ref{S:sufficiency}, we prove the sufficiency result, Theorem~\ref{T:main}.  Our proof strategy for Theorem~\ref{T:main} is the natural extension of the argument in \cite{BCCTshort}, namely, slicing along ``critical'' subspaces and induction.

\subsection*{Notation}    We use the standard notation that $A \lesssim B$ means that $A \leq CB$, where the constant $C$ may depend on $n,m$, the $L_j$, and the $p_j$.  To keep equations within lines, we will frequently abbreviate $L^p$ norms (such as $\|\cdot\|_{L^p(V_j \times \R^n \times \R)}$) with the notation $\|\cdot\|_p$.  Though the underlying measure spaces may vary, even within the same line, the space on which the integration of a particular function is to be taken will always be clear from context.

\subsection*{Acknowledgements}  This work was supported in part by NSF DMS-2246906 and the Wisconsin Alumni Research Foundation (WARF).  

\section{Examples}\label{S:examples}

In this section we introduce a few examples illustrating the sharpness, or lack thereof, in Theroem~\ref{T:main}.

\begin{exmp}\label{ex: Finner}
We first consider the natural Heisenberg analogue of the Finner case \cite{finner1992generalization} of the Brascamp--Lieb inequalities, in which the $V_j$ are coordinate subspaces.  

More precisely, for $1\leq j\leq m$, consider $M_j\coloneqq\{m_j^1,\dots,m_j^{n_j}\}\subset\{1,\dots,n\}$ and $L_j$ the orthogonal projections onto subspaces spanned by coordinates whose indices are in $M_j$, and set $L_{j+m}:=L_j$. Then elementary linear algebraic considerations show that conditions (B1) and (B2) hold for arbitrary $V \leq \R^n$ if and only if they hold for all coordinate subspaces. Furthermore, for a coordinate subspace $V$, $L_j(V^\perp) =(L_j V)^\perp$, $1 \leq j \leq m$, so conditions \eqref{C1} and \eqref{C2} coincide and imply \eqref{C} for coordinate subspaces.  (However, \eqref{C} may fail for subspaces that are not coordinate subspaces.)

By specializing a bit further, we obtain a case in which the necessary and sufficient conditions coincide.  

\begin{prop} \label{P:finner}
    Given coordinate subspaces $V_1,\ldots,V_{2m} \leq \R^n$, with $V_j=V_{j+m}$ and $\displaystyle\oplus_{j=1}^m \ker L_j^\perp=\R^n$ (recall that $V_j=L_j(\R^n)\leq\R^n$ and $V_j^\perp=\ker L_j$), inequality \eqref{E:Lp ineq sym} holds if and only if conditions \eqref{A}, \eqref{B1}, \eqref{B2}, \eqref{C} hold for all coordinate subspaces $V \leq \R^n$.
\end{prop}

(The proof of Theorem~\ref{T:main} can easily be modified to prove this result.)
\end{exmp}

\begin{exmp}\label{ex: LW}
To lay the groundwork for an example in which condition \eqref{C} is not necessary, we recall the Loomis--Whitney analogue from \cite{FOParxiv, FPmathz}, which is the special case of Example \ref{ex: Finner} when the $V_j$ are the one-dimensional coordinate subspaces.  

Let $m=n=2$, and set $V_1=V_3=\langle e_2\rangle$ and $V_2=V_4=\langle e_1\rangle$.  Thus 
$$
\begin{array}{cc}
\pi_1(x,y,t) = (x_2,y,t+\tfrac12 x_1 y_1) & \pi_2(x,y,t) = (x_1,y,t+\tfrac12 x_2 y_2)\\
\pi_3(x,y,t) = (x,y_2,t-\tfrac12 x_1 y_1) & \pi_4(x,y,t) = (x,y_1,t-\tfrac12 x_2 y_2).
\end{array}
$$
The main result of \cite{StovallMultilin} states that in order to determine the (restricted weak-type) $L^p$ inequalities for $\scriptM$, we consider vector fields tangent to the level curves of the $\pi_j$, namely, 
\[
X_j(x,y,t)=\begin{cases}
    \frac{\partial}{\partial x_j}-\frac{1}{2}y_j\frac{\partial}{\partial t},&1\leq j\leq 2,\\
    \frac{\partial}{\partial y_{j-2}}+\frac{1}{2}x_{j-2}\frac{\partial}{\partial t},&3\leq j\leq 4,
\end{cases}
\]
and we need to determine for which $j,k$ $\{X_1,X_2,X_3,X_4,[X_j,X_k]\}$ forms a frame for $\R^5$.  Since these vector fields form a frame if and only if $\{j,k\} = \{1,3\}$ or $\{j,k\} = \{2,4\}$, we determine that the extreme points of the Riesz diagram for \eqref{E:Lp ineq sym} are $(\frac{1}{5},\frac{2}{5}, \frac{1}{5}, \frac{2}{5})$ and $(\frac{2}{5},\frac{1}{5}, \frac{2}{5}, \frac{1}{5})$. This conclusion coincides with the corresponding result in \cite{FPmathz}, which also established the sharp strong-type bounds. 
\end{exmp}

\begin{exmp}\label{ex: skewed}
Perturbing the preceding example slightly, however, leads to a case in which \eqref{C} is not necessary.  

Consider $m=n=2$, $V_1=V_3 = \langle e_2 \rangle$ and $V_2=V_4=\langle e_1+e_2 \rangle$. Thus,
\begin{align*}
\pi_1(x,y,t) &= (x_2,y,t+\tfrac12 x_1 y_1) \\ \pi_2(x,y,t) &= (\tfrac12 (x_1+x_2),y,t+\tfrac14 (x_1-x_2)(y_1-y_2))\\
\pi_3(x,y,t) &= (x,y_2,t-\tfrac12 x_1 y_1) \\ \pi_4(x,y,t) &= (x,\tfrac12 (y_1+y_2),y,t-\tfrac14 (x_1-x_2)(y_1-y_2)).
\end{align*}
The following vector fields are tangent to the level curves of the $\pi_j$:  
\[
\begin{array}{rcl}
    X_1(x,y,t)&=&\frac{\partial}{\partial x_1}-\frac{1}{2}y_1\frac{\partial}{\partial t},\\
    X_2(x,y,t)&=&\frac{\partial}{\partial x_2}-\frac{\partial}{\partial x_1}+\tfrac12(y_1-y_2)\frac{\partial}{\partial t},\\
    X_3(x,y,t)&=&\frac{\partial}{\partial y_1}+\frac{1}{2}x_1\frac{\partial}{\partial t},\\
    X_4(x,y,t)&=&\frac{\partial}{\partial y_2}-\frac{\partial}{\partial y_1}-\tfrac12(x_1-x_2)\frac{\partial}{\partial t}.
\end{array}
\]
Since the kernels of $L_1$ and $L_2$ are not orthogonal complements of each other as in Loomis-Whitney case, there are more choices of $j,k$ that lead to a frame of the form $\{X_1,X_2,X_3,X_4,[X_j,X_k]\}$, namely $\{1,3\},\{1,4\}, \{2,3\}, \{2,4\}$. Thus, by \cite{StovallMultilin}, the restricted weak-type Riesz diagram for  \eqref{E:Lp ineq sym} is the convex hull of the points  $(\frac15,\frac25,\frac15,\frac25)$, $(\frac25,\frac15,\frac25,\frac15)$, $(\frac15,\frac25,\frac25,\frac15)$ and $(\frac25,\frac15,\frac15,\frac25)$. This is precisely the region determined by \eqref{A}, \eqref{B1} and \eqref{B2}, but condition \eqref{C} fails for all points of this region with $p_1 \neq p_3$.
\end{exmp}

\section{Necessity} \label{S:necessity}

This section will be devoted to the proof of the necessary conditions in Theorem~\ref{T:nec}.   

We will prove that \eqref{A}, \eqref{B1}, \eqref{B2}, \eqref{C1} and \eqref{C2} are necessary for validity of the weaker, restricted weak-type version of \eqref{E:Lp ineq sym}, that 
\begin{equation} \label{E:RWT ineq}
|\Omega| \lesssim \prod_{j=1}^{2m} |\pi_j(\Omega)|^{1/p_j}
\end{equation}
holds uniformly over bounded, open $\Omega \subseteq \HH^n$.  (That \eqref{E:Lp ineq sym} implies \eqref{E:RWT ineq} may be seen by approximating the $f_j:=\pi_j(\Omega)$ by continuous functions.)  We make the side note that necessity of  $p_j \geq 1$ can be seen by considering, for each $j$, 
$$
\Omega_j:=B^{2n+1}_1(0)\cap\pi_j^{-1}(B^{n_j+n}_{\varepsilon}(0)\times (-1,1)),
$$
where $B^d_\varepsilon(0)$ is the $\eps$-ball centered at zero in $\R^d$.  

\begin{lemma}[Scaling condition]
If inequality \eqref{E:RWT ineq} holds, then condition \eqref{A} holds.  
\end{lemma}
\begin{proof}  We take 
$$
\Omega:=\{(x,y,t)\in\R^{2n+1}:|x|\leq r,|y|\leq1,|t|\leq r\},
$$
with $r > 0$. Then $|\Omega|\sim r^{n+1}$, while 
$$
|\pi_j(\Omega)|\sim\begin{cases}
    r^{n_j+1},&1\leq j\leq m,\\
    r^{n+1},&m+1\leq j\leq 2m,
\end{cases}$$
where $n_j\coloneqq\dim V_j:=L_j(\R^n)$ for each $1 \leq j \leq m$. Assuming \eqref{E:RWT ineq}, we have 
$$
r^{n+1}\lesssim r^{A(\mathbf{p},\R^n)}, \qquad A(\mathbf p,\R^n):= \sum_{j=1}^m \tfrac{n_j + 1}{p_j}+\tfrac{n+1}{p_{m+j}}.
$$
Letting $r\to\infty$ and $r\to 0^+$, we obtain
\begin{equation} \label{E:A 1}
n+1=\displaystyle\sum_{j=1}^m\tfrac{n_j+1}{p_j}+\tfrac{n+1}{p_{j+m}}.
\end{equation}
Applying the analogous argument for the set 
$$
\Omega:=\{(x,y,t) \in \HH^n : |x| \leq 1, |y| \leq r, |t| \leq r\},
$$
yields the other identity in \eqref{A},
\begin{equation} \label{E:A 2}
n+1=\displaystyle\sum_{j=1}^m\tfrac{n+1}{p_j}+\tfrac{n_j+1}{p_{m+j}}.
\end{equation}
\end{proof}

\begin{lemma}
If inequality \eqref{E:RWT ineq} holds, then conditions \eqref{B1} and \eqref{B2} hold.  
\end{lemma}
\begin{proof}
Let 
$$
\Omega:=\{(x,y,t)\in\R^{2n+1}:|x_V|\leq R, |x_V^\perp| \leq 1, |y|\leq 1,|t|\leq R\},
$$
for $V\leq\R^n$ and $R \geq 1$, where $x_V$ denotes the orthogonal projection of $x$ onto $V$.  Then $|\Omega|\sim R^{\dim V+1}$ and
$$|\pi_j(\Omega)|\sim\begin{cases}
    R^{\dim L_jV+1},&1\leq j\leq m,\\
    R^{\dim V+1},&m+1\leq j\leq 2m.
\end{cases}$$
Inequality \eqref{E:RWT ineq} implies
$$
R^{\dim V+1}\lesssim R^{A(\mathbf p,V)}, \qquad A(\mathbf p,V):= \sum_{j=1}^m\tfrac{\dim L_jV+1}{p_j}+\tfrac{\dim V+1}{p_{j+m}}.
$$
Letting $R\to\infty$, we obtain \eqref{B1},
and by symmetry in $x,y$, we also obtain \eqref{B2}.  
\end{proof}

\begin{lemma}
If inequality \eqref{E:RWT ineq} holds, then for any subspaces $V\leq\R^n$ and $W\leq V^\perp$, \eqref{C1} and \eqref{C2} hold.
\end{lemma}

\begin{proof}
Let $V \leq \R^n$ and $W\leq V^\perp$, and set
$$
\Omega:=\{(x,y,t)\in\R^{2n+1}:|x_V|\leq R, |y_{W^\perp}|\leq R^{-1},|x_V^\perp|,|y_W| \leq 1, |t|\leq 1\}.
$$
It is immediate that $|\Omega|\sim R^{\dim V-\dim W^\perp}$.  Since $\hat L_jx\cdot\hat L_jy=x\cdot y-L_jx\cdot L_jy$, by change of variables, the projections of $\Omega$ have the same volume as the projections via 
$$
\tilde\pi_j(x,y,t)=\begin{cases}
    (L_jx,y,t+\frac{1}{2}x\cdot y), & 1\leq j\leq m, \\
    (x,L_jy,t-\frac{1}{2}x\cdot y), & m<j\leq2m.
\end{cases}
$$
Now for $R>1$,  
$$
|\pi_j(\Omega)|\lesssim\begin{cases}
    R^{\dim L_jV-\dim W^\perp},&1\leq j\leq m\\
    R^{\dim V-\dim (L_jW)^\perp},&m+1\leq j\leq 2m.
\end{cases}$$
Thus, by \eqref{E:RWT ineq}, sending $R\nearrow \infty$ yields 
 \eqref{C1}; \eqref{C2} follows by symmetry.  
\end{proof}

\section{Sufficiency} \label{S:sufficiency}   
In this section, we adapt the induction on critical subspaces argument of \cite{BCCTshort} to establish sufficiency of conditions \eqref{A}, \eqref{B1}, \eqref{B2} and \eqref{C}.

First of all, it is worth noting that conditions \eqref{B1}, \eqref{B2} and \eqref{C} are equivalent to
\begin{equation}
\tag{B}\label{B}
\begin{aligned}
    \dim V+1  &\leq  \sum_{j=1}^{m}\frac{\dim L_jV+1}{p_j}+\frac{\dim V+1}{p_{j+m}}
    \\
    &=
\sum_{j=1}^{m}\frac{\dim V+1}{p_j}+\frac{\dim L_jV+1}{p_{j+m}}
\end{aligned}
\end{equation}
for all $V\leq\R^n$.

In fact, we will prove a slightly more general result, which immediately implies the sufficiency portion of Theorem~\ref{T:main}.  Given $\mathbf a,\mathbf b \in \prod_{j=1}^{2m} V_j^\perp$ (where we recall $V_j=L_j(\R^n)$),  we define for $1 \leq j \leq m$
$$
\begin{aligned}
 \pi_j^{\mathbf a,\mathbf b}&:=(L_j(x),y,t+\frac{1}{2}(\hat{L}_j(x)+a_j)\cdot(\hat{L}_j(y)+b_j)),\\
\pi_{j+m}^{\mathbf a,\mathbf b}&:=(x,L_j(y),t-\frac{1}{2}(\hat{L}_j(x)+a_{j+m})\cdot(\hat{L}_j(y)+b_{j+m})) ,
\end{aligned}
$$
and the corresponding multilinear form
\begin{equation} \label{E:Mab}
\scriptM^{\mathbf a, \mathbf b} (f_1,\ldots,f_{2m}):= \int_{\HH^n} \prod_{j=1}^{2m} f_j \circ \pi_j^{\mathbf a, \mathbf b}(x,y,t)\, dx\, dy\, dt.
\end{equation}
This section will be devoted to a proof of the following.  
\begin{prop} \label{P:gen}
Let $n,m \geq 1$, $\mathbf p \in [1,\infty]^{2m}$ and projections $L_j:\R^n \to V_j \leq \R^n$, $1 \leq j \leq m$, be given, and assume that conditions \eqref{A} and \eqref{B} hold for all $V \leq \R^n$.  Then 
\begin{equation} \label{E:Lp gen}
\scriptM^{\mathbf a,\mathbf b}(f_1,\ldots,f_{2m}) \lesssim \prod_{k=1}^{2m} \|f_k\|_{p_k}
\end{equation}
holds uniformly over all nonnegative, continuous, compactly supported $f_k$, with implicit constant independent of the constants $\mathbf a,\mathbf b$.
\end{prop}

We note that various changes of variables lead to a number of equivalent versions of \eqref{E:Lp gen}. 
 We have chosen the formulation that is superficially the most general and hence the easiest to apply.   

Because the choice of $\mathbf a, \mathbf b$ does not change the conditions \eqref{A} and \eqref{B}, by interpolation, it suffices to prove \eqref{E:Lp gen} when 
$$
\mathbf p:=(p_1^{-1},\ldots,p_{2m}^{-1})
$$
is an extreme point of the set $Q$ defined by
$$Q\coloneqq\{(t_1,\dots,t_{2m})\in[0,1]^{2m}:(t_1^{-1},\dots,t_{2m}^{-1})\text{ satisfies conditions \eqref{A} and \eqref{B}}\};$$
in this case we say that $\mathbf p$ is extreme. Since $Q$ is compact and convex, it may be expressed as the convex hull of its extreme points.

If $\mathbf p$ is extreme, then
\begin{enumerate}
    \item The inequality in  \eqref{B} is equality for at least one proper subspace $V < \R^n$, or
    \item $p_j\in\{1,\infty\}$ for at least one index $1\leq j\leq 2m$.
\end{enumerate}
In particular, we say that a subspace $V \leq \R^n$ is \textit{critical} if equality holds in (B).  When $W < \R^n$ is critical, we may decompose the integration on the right of  \eqref{E:Mab} as 
\begin{equation}\label{E:slicing}
\iint_{W^\perp\times W^\perp}\left(\iiint_{W\times W \times \R} \prod_{j=1}^{2m} f_j \circ \pi_j^{\mathbf a, \mathbf b}\, dx_W\, dy_W\, dt\right)dx_{W^\perp}\,dy_{W^\perp},
\end{equation}
and the inductive step will involve applying the Euclidean Brascamp--Lieb inequalities \cite{BCCTshort, BCCTlong} to the outer double integral and Proposition~\ref{P:gen} (in a lower dimensional case) to the inner triple integral.  (It is in this slicing step that we are led to consider the generalization $\scriptM^{\mathbf a, \mathbf b}$.)

Let's recall the Euclidean Brascamp--Lieb inequalities. For each $j$, we define $\pi_j^\flat$ to be the linear mapping on $\R^{2n}$ whose value at $(x,y)$ is obtained by omitting the last coordinate of $\pi_j(x,y,0)$.  That is, 
$$
\pi_j^\flat(x,y) := (L_j x,y), \qquad \pi_{j+m}^\flat(x,y) := (x,L_jy), \qquad 1 \leq j \leq m.
$$

The associated multilinear form is 
$$
\scriptM^\flat(f_1,\ldots,f_{2m}) = 
\iint \prod_{j=1}^{2m} f_j \circ \pi_j^\flat(x,y)\, dx\, dy,
$$
and it will be useful later to know when the inequality
\begin{equation} \label{E:Mflat ineq}
\scriptM^\flat(f_1,\ldots,f_{2m}) \lesssim \prod_{j=1}^{2m}\|f_j\|_{p_j}
\end{equation}
holds.  Indeed, the following proposition may be deduced directly the main result of from \cite{BCCTshort} by integrating over each copy of $\R^n$ separately. 

\begin{prop}[\cite{BCCTshort, BCCTlong}] \label{P:linear BL}
    Inequality \eqref{E:Mflat ineq} holds if and only if all of the conditions below hold: 
\begin{gather}
   \tag{$A^\flat$}\label{Aflat} n=\sum_{j=1}^{2m}\frac{\dim V_j}{p_j}+\frac{n}{p_{j+m}}=\sum_{j=1}^{2m}\frac{\dim n}{p_j}+\frac{\dim V_j}{p_{j+m}};\\
   \tag{$B1^\flat$}\label{B1flat} \dim V \leq \sum_{j=1}^{m} \frac{\dim L_jV}{p_j}+\frac{\dim V}{p_{j+m}}, \qquad V\leq\R^n;\\
   \tag{$B2^\flat$}\label{B2flat} \dim V \leq \sum_{j=1}^{m} \frac{\dim V}{p_j}+\frac{\dim L_jV}{p_{j+m}}, \qquad V\leq\R^n.
\end{gather}
\end{prop}

Turning to the induction on critical subspaces argument for Proposition~\ref{P:gen}, our base case occurs when $\mathbf p$ is extreme and either $\{0\}$ is critical or there are no proper critical subspaces $W < \R^n$.  We begin by analyzing these cases.

\begin{lemma} \label{L:0 critical}
Inequality \eqref{E:Lp gen} holds if $\{0\}$ is critical.
\end{lemma}

\begin{proof}
Since $\{0\}$ is critical, we have
$\sum_{j=1}^{2m} \tfrac1{p_j} = 1$.
Inserting this condition into \eqref{A} yields
$$n\displaystyle\sum_{j=1}^{2m}\frac{1}{p_j}=\displaystyle\sum_{j=1}^m\frac{n_j}{p_j}+\frac{n}{p_{j+m}}=\displaystyle\sum_{j=1}^m\frac{n}{p_j}+\frac{n_j}{p_{j+m}}.$$
Thus, for each $1\leq j\leq m$, we have $p_j=p_{j+m}=\infty$ or $n_j=n$.  In the former case, \eqref{E:Lp gen} is reduced to the case of $m-1$ pairs of projections.  In the latter case, $L_j = \mathbb \I_n$ and $\hat L_j=0$, so $\pi_j$ is the identity map, and \eqref{E:Lp gen} follows from H\"older's inequality.
\end{proof}

\begin{lemma} \label{L:base cases}
    If $\mathbf p$ is extreme, and no proper subspace $W<\R^n$ is critical, then one of the following must occur:  
    \begin{enumerate}[i.]
        \item For some $1 \leq j, j' \leq m$, $p_j=p_{j'+m} = \infty$ and $L_j=L_{j'}$; or 
        \item $m=1$, $L_1=0$, and $p_1=p_2=\frac{n+2}{n+1}$ .
    \end{enumerate}
\end{lemma}

\begin{proof}
We first assume that $m = 1$.  If $L_1=0$, then by applying \eqref{A}, we see that Conclusion ii holds. If $\text{rank }L_1=n$, then $\tfrac1{p_1}+\tfrac1{p_2}=1$ from \eqref{A}; that is, $\{0\}$ is a critical subspace, contradicting our assumption. Finally, if $0 < \text{rank } L_1 < n$, then taking $V:=\ker L_1<\R^n$ in \eqref{A}, we obtain $p_1=p_2=2-\frac{\dim V}{n+1}$. Furthermore, \eqref{B} yields
$$\dim V+1<\frac{\dim V+2}{p_1},$$
which, inserting the values of $p_1,p_2$ that we just obtained, becomes
$$2-\frac{\dim V}{n + 1} = p_1 < \frac{\dim V + 2}{\dim V + 1} = 2-\frac{\dim V}{\dim V + 1},$$
    contradicting $0<\dim V < n$.

To summarize, Conclusion ii holds when $m=1$.

Now we assume $m>1$. Because \eqref{C} may be rewritten as a system of equations for $p_j^{-1}-p_{j+m}^{-1}$, and the second equation in \eqref{A} is \eqref{C} in the case $V=\R^n$,  \eqref{A} and \eqref{C} collectively determine at most $m+1$ independent linear equations.  Since we have assumed that $\mathbf p$ is extreme and that equality never holds in \eqref{B}, we must have $p_j\in\{1,\infty\}$ for $l\geq 2m-(m+1)=m-1$ values of $j$. However,  $p_j=1$ cannot hold.  Indeed, if $p_k=1$ for some $1 \leq k \leq 2m$, then by \eqref{A} and the fact that $\mathbf p\in[1,\infty]^{2m}$, we have $p_j=\infty$ for all $1 \leq j\neq k \leq 2m$, whence,  $\sum_{j=1}^{2m}1/p_j=1/p_k=1$, so $\{0\}$ is critical, contradicting our assumption that \eqref{B} never holds.  

Now we have that $p_j=\infty$ for $l \geq m-1 \geq 1$ values of $j$.  If $l \geq m+1$, then $p_j=p_{j+m}=\infty$ for some $1 \leq j \leq m$ by pigeonholing, so Conclusion i holds with $j=j'$.  If $l=m-1$, \eqref{A} and \eqref{C} together contain exactly $m+1$ linearly independent equations.  In particular, \eqref{C} contains $m$ independent homogeneous equations for $p_j^{-1}-p_{j+m}^{-1}$, whence $p_j=p_{j+m}$ for all $j$.  Since $p_j=\infty$ for some $j$, Conclusion i holds with $j=j'$.  

We turn now to the last case, that $l=m$ and exactly half of the $p_j$ are infinite. If $p_j=p_{j+m}$ for some $j$, then, by pigeonholing, for some (possibly different) $j'$, it must hold that $p_{j'} = p_{j'+m}=\infty$, so we are in Conclusion i. 

Now we may assume that exactly one of $p_j$, $p_{j+m}$ equals infinity for every $1 \leq j \leq m$.  Suppose that the $L_j$, $1 \leq j \leq m$, are distinct; equivalently, their kernels $K_j:=\ker L_j = V_j^\perp$ are distinct. Reordering indices if necessary, we may assume that $K_1$ is not contained in any $K_j$ with $1 < j \leq m$.  Therefore $K_1 \not \subset \bigcup_{j=2}^m K_j$.  Let $v \in K_1 \setminus \bigcup_{j=2}^m K_j$, and consider $V:=\langle v \rangle$, the span of $v$.  Applying \eqref{C} with this $v$, $p_1=p_{1+m}$, but this contradicts our assumption that exactly one of each $p_j,p_{j+m}$ is infinite.  Tracing back, the $L_j$ cannot all be distinct.  

Finally, we may assume that $L_r=L_s$ for some $1 \leq r \neq s \leq m$.  The equations \eqref{A} and \eqref{C} cannot distinguish $p_r^{-1}$ from $p_s^{-1}$, so for $\mathbf p$ to be extreme in the absence of \eqref{B}, it must hold that $p_r=\infty$ or $p_s=\infty$ (since neither can equal 1).  Similarly, it must also hold that $p_{r+m} = \infty$ or $p_{s+m} = \infty$.  This once again puts us in the situation described by Conclusion i, completing the proof of the Lemma.  
\end{proof}

To complete the proof that Proposition~\ref{P:gen} holds in the base case of our induction argument, it remains to prove that \eqref{E:Lp gen} holds under Conclusions i and ii from  Lemma~\ref{L:base cases}.  In fact, it suffices to consider only Conclusion ii.  Indeed, in Conclusion i, after reindexing, we may assume that $j=j'$.  By H\"older's inequality, inequality \eqref{E:Lp ineq sym} is equivalent to the analogous inequality for a form with a lower order of multilinearity and exponents that still obey (A) and (B), and the estimate is obtained by induction on the order of multilinearity.  

\begin{lemma}\label{caseiiproof}
    Inequality \eqref{E:Lp gen} holds under Conclusion ii of  Lemma~\ref{L:base cases}.  
\end{lemma}
\begin{proof}
By assumption,
$$\pi_j^{\mathbf a,\mathbf b}(x,y,t)=\begin{cases}
    (y,t+\frac{1}{2}(x+a_1)\cdot(y+b_1)),&j=1,\\
    (x,t-\frac{1}{2}(x+a_1)\cdot(y+b_1)),&j=2.
\end{cases}$$
Therefore, after a change of variables,
$$
\scriptM^{\mathbf a, \mathbf b}(f_1,f_2) =\bigintsss_{\R^{2n+1}}f_1(y-b_1,t+\frac{1}{2}x\cdot y)\,f_2(x-a_1,t-\frac{1}{2}x\cdot y)\,dx\,dy\,dt.
$$
By a further change of variables, it suffices to consider the case $\mathbf a = \mathbf b = 0$.  In summary, we wish to prove the inequality
\begin{equation} \label{E:radon}
    \int_{\R^{2n+1}} f_1(y,t+\tfrac12 x \cdot y)\, f_2(x,t-\tfrac12 x \cdot y)\, dx\, dy\, dt \lesssim \|f_1\|_{\frac{n+2}{n+1}}\|f_2\|_{\frac{n+2}{n+1}},
\end{equation}
which is well-known. 
Indeed, \eqref{E:radon} is  better known as the endpoint bound for the Radon transform \cite{Calderon83, oberlin1982mapping}, or, equivalently \cite{christ2014extremizers}, the endpoint bound for convolution with affine surface measure on the paraboloid \cite{Littman73}.  
\end{proof}

Finally, we turn to the inductive step initiated with the decomposition \eqref{E:slicing}.

\begin{lemma}\label{L:ind step}
If $\{0\} \neq W < \R^n$ is a proper critical subspace and Proposition~\ref{P:gen} holds with $n$ replaced by $\dim W$, then \eqref{E:Lp gen} holds for the given data.      
\end{lemma}

\begin{proof}
Let $\{0\} \neq W < \R^n$ be a critical subspace.  For $x \in \R^n$ and $\xi_j \in V_j$, we write $x=x'+x''$ and $\xi_j=\xi_j'+\xi_j''$, where 
$$
x':=\rm{pr}_W x, \qquad \xi_j':=\rm{pr}_{W_j} \xi_j,
$$
where $W_j:=L_j W$, $j=1,\dots,m$.  We introduce an auxillary multilinear form, 
\begin{align*}
    &\scriptM^{\mathbf a, \mathbf b, x'',y''}(g_1,\ldots,g_{2m})\\
    &\quad := 
    \iiint_{W^2 \times \R} \prod_{j=1}^m g_j(L_j x'+(L_jx'')',y',t+\tfrac12(\hat L_j x'+\mathfrak a_j) \cdot (\hat L_j y'+\mathfrak b_j)) \\
    &\qquad \times g_{j+m}( x',L_{j+m} y'+ (L_{j+m} y'')',t-\tfrac12(\hat L_{j+m} x'+\mathfrak a_{j+m}) \cdot (\hat L_{j+m} y'+\mathfrak b_{j+m}))\\
    &\qquad dx'\,dy'\,dt,
\end{align*}
where $\mathfrak a_j:= a_j + \hat L_j x''$ and $\mathfrak b_j := b_j+\hat L_j y''$ for $j=1,\dots,2m$.  Thus the inner (triple) integral in \eqref{E:slicing} equals $\scriptM^{\mathbf a, \mathbf b, x'',y''}(g_1,\ldots,g_{2m})$, with 
\begin{equation} \label{E:gk}
\begin{aligned}
g_j &= g_j^{{\L}_j x'',y''}(\xi_j',y',s):=f_j(\xi_j' + {\L}_j x'',y'+y'',s),\\
g_{j+m} &= g_{j+m}^{x'',{\L_{j+m}}y''}(x',\eta_j',s):=f_{j+m}(x' +  x'',\eta_j'+{\L}_{j+m} y'',s), \qquad 1 \leq j \leq m,
\end{aligned}
\end{equation}
where $\L_j = \rm{pr}_{W_j^\perp} L_j|_{W^\perp}$, $j=1,\dots,m$.

We claim that 
\begin{equation} \label{E:Mabx''y''}
\scriptM^{\mathbf a, \mathbf b, x'', y''}(g_1,\ldots,g_{2m}) \lesssim \prod_{j=1}^m \|g_j\|_{L^{p_j}(W_j \times W \times \R)}\|g_{j+m}\|_{L^{p_{m+j}}(W \times W_{j+m} \times  \R)},
\end{equation}
for (e.g.) continuous, nonnegative $g_j$, with implicit constant independent of $\mathbf a, \mathbf b$ and $x'',y''$.  By rotating coordinates, we may assume that $W=\R^{n'} \times \{0\}$ (henceforth identified with $\R^{n'}$).  For each $j=1,\dots,m$, there exist invertible linear maps $A_j$, depending only on the $L_j$ and $W$, such that $A_j \circ L_j|_W = \mathfrak L_j$ is a projection map.  Thus (by translation invariance of $L^p$ norms), it suffices to prove \eqref{E:Mabx''y''} with $\scriptM^{\mathbf a, \mathbf b, x'', y''}$ replaced by 
\begin{align} \label{E:new form}
&\int_{\HH^{n'}}\prod_{j=1}^m g_j (\mathfrak L_j u,v,t+\tfrac12(\hat{ L}_ju + \mathfrak a_j) \cdot (\hat{ L}_jv + \mathfrak b_j))\\\notag
 &\qquad\qquad \times g_{j+m} ( u,\mathfrak L_{j+m} v,t-\tfrac12(\hat{ L}_{j+m} u + \mathfrak a_{j+m}) \cdot (\hat{ L}_{j+m} v+ \mathfrak b_{j+m}))\,du\,dv\,dt,
\end{align}
where  $\mathfrak a_j$ depends on $a_j, x'', L_j, W$ (analogously for $\mathfrak b_j$) for $j=1,\dots,2m$.  Since the kernels of $L_j$ and $\hat L_j$ have trivial intersection, the same holds for the restriction of these maps to $W$.  Moreover, after a further coordinate change in the $g_j$ (with constant Jacobian determinant depending only on the $L_j$ and $W$), we may replace the $\hat L_j|_W$ in \eqref{E:new form} with $\hat{\mathfrak L}_j := \mathbf I_{n'}-\mathfrak L_j$.  The resulting multilinear form now takes the form \eqref{E:Mab}, and, furthermore, condition (A) for this form follows by the assumption that $W$ is critical, and (B) and (C) for the new form are directly inherited by those conditions for the original form for $V \leq W$.   We may now apply the inductive hypothesis to $\scriptM^{\mathbf a, \mathbf b, x'',y''}$ to conclude the claimed inequality \eqref{E:Mabx''y''}.  

Inserting this estimate into \eqref{E:slicing}, 
\begin{equation} \label{E:outer form}
\scriptM^{\mathbf a, \mathbf b}(f_1,\ldots,f_m) \lesssim 
\iint_{W^\perp \times W^\perp} \prod_{j=1}^m F_j({\L}_j x'',y'')\,F_{j+m}(x'',{\L}_{j+m}y'')\, dx''\, dy'',
\end{equation}
where
$$
F_j(\xi_j'',y''):=\|g_j^{\xi_j'',y''}\|_{L^{p_j}(W_j \times W \times \R)}, \quad F_{j+m}(x'',\eta_j''):=\|g_{j+m}^{x'',\eta_j''}\|_{L^{p_{j+m}}(W \times W_{j+m} \times \R)},
$$
with the $g_k$ defined as in \eqref{E:gk}.  

Finally, we wish to apply Proposition~\ref{P:linear BL} to the right hand side of \eqref{E:outer form}.  Indeed, conditions \eqref{Aflat}, \eqref{B1flat}, and \eqref{B2flat} for the maps ${\L}_j$ and vector spaces $U \leq W^\perp$ may be deduced by writing conditions \eqref{A} and \eqref{B} with maps $L_j$ and vector spaces $V=W+U$ and then subtracting the equality case, \eqref{B} with $V=W$.  

We obtain
$$
\scriptM^{\mathbf a, \mathbf b}(f_1,\ldots,f_{2m}) \lesssim \prod_{j=1}^m \|F_j\|_{L^{p_j}(W_j^\perp \times W)}\|F_{j+m}\|_{L^{p_{j+m}}(W^\perp \times W_j^\perp)},
$$
and \eqref{E:Lp gen} follows by Fubini.  
\end{proof}

%%%%%%%%%%%%%%%%%%%%%%%%%%%%%%%%%%%%%%%%%%%%%%%
%%%%%%%%%%%%%%%%%%%%%%%%%%%%%%%%%%%%%%%%%%%%%%%
%%%%%%%%%%%%%%%%%%%%%%%%%%%%%%%%%%%%%%%%%%%%%%%

%%%%%%%%%%%%%%%%%%%%%%%%%%%%%%%%%%%%%%%%%%%%%%%
%%%%%%%%%%%%%%%%%%%%%%%%%%%%%%%%%%%%%%%%%%%%%%%
%%%%%%%%%%%%%%%%%%%%%%%%%%%%%%%%%%%%%%%%%%%%%%%

\end{document}